\documentclass[12pt,leqno]{amsart}
\usepackage{amssymb,amsmath,amsfonts,amsbsy, amsthm, xspace, latexsym, amscd, graphicx, fancyhdr,verbatim,enumitem}
\usepackage[colorlinks=true, pdfstartview=FitV, linkcolor=blue, citecolor=blue, urlcolor=blue]{hyperref}

\swapnumbers
\theoremstyle{definition}
\newtheorem{definition}{Definition}[section]

\newtheorem{remark}[definition]{Remark}

\newtheorem{example}[definition]{Example}
\newtheorem{examples}[definition]{Examples}

\theoremstyle{plain}
\newtheorem{theorem}[definition]{Theorem}

\newtheorem{lemma}[definition]{Lemma}

\newcommand{\bra}{\ensuremath{\big\langle}}
\newcommand{\ket}{\ensuremath{\big\rangle}}
\newcommand{\supp}{\ensuremath{{\rm{supp}}}}
\newcommand{\card}{\ensuremath{{\rm{card}}}}

\newcommand{\st}{\ensuremath{{{\rm{st}}}}}

\newcommand{\emb}{\ensuremath{\hookrightarrow}}
\newcommand{\rest}{\!\!\ensuremath{\upharpoonright}\!}

\newcommand{\ifff}{\ensuremath{{\rm{\,iff\;}}}}

\newlist{T-enum}{enumerate}{2}
\newlist{L-enum}{enumerate}{2}
\newlist{C-enum}{enumerate}{2}
\newlist{P-enum}{enumerate}{2}  % propositions
\newlist{Pf-enum}{enumerate}{2} % proofs
\newlist{D-enum}{enumerate}{2}
\newlist{Ex-enum}{enumerate}{2}
\newlist{Exs-enum}{enumerate}{2}
\newlist{E-enum}{enumerate}{2}
\newlist{R-enum}{enumerate}{2}
\setlist[T-enum,1]{label=(\roman*),format=\bfseries\emph,leftmargin=*,labelindent=.1\parindent}
\setlist[T-enum,2]{label=(\alph*),format=\bfseries\emph,leftmargin=*,labelindent=.1\parindent}
\setlist[L-enum,1]{label=(\roman*),format=\bfseries\emph,leftmargin=*,labelindent=.1\parindent}
\setlist[L-enum,2]{label=(\alph*),format=\bfseries\emph,leftmargin=*,labelindent=.1\parindent}
\setlist[C-enum,1]{label=(\roman*),format=\bfseries\emph,leftmargin=*,labelindent=.1\parindent}
\setlist[C-enum,2]{label=(\alph*),format=\bfseries\emph,leftmargin=*,labelindent=.1\parindent}
\setlist[P-enum,1]{label=(\roman*),format=\bfseries\emph,leftmargin=*,labelindent=.1\parindent}
\setlist[P-enum,2]{label=(\alph*),format=\bfseries\emph,leftmargin=*,labelindent=.1\parindent}
\setlist[Pf-enum,1]{label=(\roman*), leftmargin=*,labelindent=.1\parindent}
\setlist[Pf-enum,2]{label=(\alph*), leftmargin=*,labelindent=.1\parindent}
\setlist[D-enum,1]{label=\textbf{\arabic*.},leftmargin=*,labelindent=.2\parindent}
\setlist[D-enum,2]{label=\textbf{(\alph*)},leftmargin=*,labelindent=.1\parindent}
\setlist[Ex-enum,1]{label=\textbf{\arabic*.},leftmargin=*,labelindent=.15\parindent}
\setlist[Exs-enum,1]{label=\textbf{\arabic*.},leftmargin=*,labelindent=.15\parindent}
\setlist[E-enum,1]{label=\textbf{\arabic*.},leftmargin=*,labelindent=.15\parindent}
\setlist[R-enum,1]{label=\textbf{\arabic*.},leftmargin=*,labelindent=.15\parindent}
\begin{document}
%________________________________
\title{Algebraic Approach to Colombeau Theory}
%___________________________________
\author{Todor D. Todorov*} 
\address{Mathematics Department\\                
California Polytechnic State University\\
San Luis Obispo, California 93407, USA}
\email{ttodorov@calpoly.edu}
\thanks{*Partly supported by a FWF (Austrian Science Fund) project P 25311-N25: \emph{Non-Archimedean Geometry and Analysis}}


.
\date{September 2013}

%____________________________
\begin{abstract} We present a differential algebra of generalized functions over a field of generalized scalars by means of several axioms in terms of general algebra and topology. Our differential algebra is of Colombeau type in the sense that it contains a copy of the space of Schwartz distributions, and the set of regular distributions with $\mathcal C^\infty$-kernels forms a differential subalgebra. We discuss the uniqueness of the field of scalars as well as the consistency and independence of our axioms. This article is written mostly to satisfy the interest of mathematicians and scientists who do not necessarily belong to the \emph{Colombeau community}; that is to say, those who do not necessarily work in the \emph{non-linear theory of generalized functions}. 
\end{abstract}
\maketitle

\section{Introduction}

Our algebraic approach has \emph{three main goals}:
\begin{enumerate}

\item To \emph{improve the properties of the generalized scalars}: In our approach the set of scalars (the constant functions) of our algebra of generalized functions forms an algebraically closed Cantor complete field, not a ring with zero divisors as in the original Colombeau theory (Colombeau~\cite{jCol84a}-\cite{jCol91}, Oberguggenberger~\cite{mOber92}, Biagioni~\cite{aBiag90}).

\item To transfer \emph{more general theoretical tools} from functional analysis to Colombeau theory. In particular, the validity of the Hahn-Banach extension principle in our approach is a direct consequence of the improvement of the scalars (Todorov \& Vernaeve~\cite{TodVern08}, Section 8).

\item To improve the \emph{accessibility} of the theory outside the ``Colombeau Community''. Most of our axioms are algebraic in nature (hence, the title ``Algebraic Approach...'') and can be comprehended without preliminary knowledge of Colombeau theory. 
\end{enumerate}

	In our algebraic approach \emph{we follow the familiar examples of real and functional analysis}: The field of real numbers $\mathbb R$ is defined axiomatically; $\mathbb R$ is a \emph{totally ordered complete field}. These axioms characterize the field $\mathbb R$ uniquely - up to isomorphism. The existence of $\mathbb R$ is guaranteed by a construction of a model of $\mathbb R$ in terms of the rationals $\mathbb Q$. We supply $\mathbb R$ with the order topology. Next, we define the field of the complex numbers by $\mathbb C=\mathbb R(i)$ and supply $\mathbb C$  with the corresponding product topology. Let $\Omega$ be an open subset of $\mathbb R^d$. The space $\mathcal E(\Omega)=\mathcal C^\infty(\Omega)$ is consists of the functions $f: \Omega\to \mathbb C$ with continuous partial derivatives of any order.

 Similarly, in order to define the \emph{algebra of generalized functions} $\widehat{\mathcal E(\Omega)}$ (of Colombeau type) \emph{over the field of scalars} $\widehat{\mathbb C}$ first, we define the field of real generalized scalars $\widehat{\mathbb R}$ by several axioms which determine $\widehat{\mathbb R}$ uniquely up to a field isomorphism. Since $\widehat{\mathbb R}$ is a real closed (thus, totally ordered) field, we can topologize $\widehat{\mathbb R}$ with the order topology. The field of scalars in our approach is $\widehat{\mathbb C}=\widehat{\mathbb R}(i)$. Let $\Omega$ be an open set of $\mathbb R^d$. We extend $\Omega$ to its \emph{monad}
$
\widehat{\mu}(\Omega)=\big\{r+h: r\in\Omega,\; h\in{\widehat{\mathbb R}^d},\; |h|\approx 0\big\}$,
which is used as the common domain of the generalized functions. (Here $|h|\approx 0$ means that $|h|$ is infinitesimal.) Let $\mathcal C^\infty(\widehat{\mu}(\Omega),\; \widehat{\mathbb C})$ consists of all functions $f: \widehat{\mu}(\Omega)\to \widehat{\mathbb C}$ with continuous partial derivatives of all orders. The set $\mathcal C^\infty(\widehat{\mu}(\Omega),\; \widehat{\mathbb C})$ is a \emph{differential algebra over the field} $\widehat{\mathbb C}$, but it is too large for developing any basic calculus: neither the intermediate value theorem, nor the fundamental theory of calculus hold in $\mathcal C^\infty(\widehat{\mu}(\Omega),\; \widehat{\mathbb C})$. Next, we select a differential subalgebra $\widehat{\mathcal E(\Omega)}$ of $\mathcal C^\infty(\widehat{\mu}(\Omega),\; \widehat{\mathbb C})$ with several additional axioms including the \emph{mean value theorem} (treated as an axiom) and the existence of a \emph{Colombeau type of embedding} $E_\Omega: \mathcal D^\prime(\Omega)\to \widehat{\mathcal E(\Omega)}$ of the space of Schwartz distributions  $\mathcal D^\prime(\Omega)$ (Vladimirov~\cite{vVladimirov}). Thus $\widehat{\mathcal E(\Omega)}$ converts to an \emph{algebra of generalized functions of Colombeau type} (Colombeau~\cite{jCol84a}-\cite{jCol91}, Oberguggenberger~\cite{mOber92}, Biagioni\cite{aBiag90}). The consistency of our axioms is proved by models - one in the framework of Robinson non-standard analysis (Oberguggenberger \& Todorov~\cite{OberTod98}) and another in the framework of standard analysis (Todorov \& Vernaeve~\cite{TodVern08}). At the end of the article we also discuss the \emph{partial independence} of our axioms.

	This article is an improved and simplified version of the axiomatic approach in (Todorov~\cite{tdTodorovAxioms11}). Our set-theoretical framework is the usual ZFC axioms in set theory together with the axiom: $2^\mathfrak{c}=\mathfrak{c}^+$, where $\mathfrak{c}=\card(\mathbb{R})$ (known as \emph{General Continuum Hypothesis} or GCH for short). Here $\mathfrak{c}^+$ stands for the cardinal number which is the successor of $\mathfrak{c}$. The only role of the GCH is to guarantee the uniqueness of the field of scalars $\widehat{\mathbb C}$; the readers who are not particularly concerned about the uniqueness of $\widehat{\mathbb C}$ might decide to ignore this axiom.
	
	Let $\Omega$ be an open subset of $\mathbb R^d$. We denote by $\mathcal E(\Omega)$ the differential ring of the $\mathcal C^\infty$-functions from $\Omega$ to $\mathbb C$. We denote by $\mathcal D(\Omega)$ the space of all functions in $\mathcal E(\Omega)$ with compact support in $\Omega$. Next, $\mathcal D^\prime(\Omega)$ stands for the space of Schwartz distributions on $\Omega$.  We denote by $\mathcal L_{loc}(\Omega)$ the set of the locally integrable (by Lebesgue) functions from $\Omega$ to $\mathbb C$. Finally, we denote by $S_\Omega:\mathcal L_{loc}(\Omega)\to\mathcal D^\prime(\Omega)$ the \emph{Schwartz embedding} defined by  $\bra S_\Omega(f), \tau\ket=\int_\Omega\, f(x)\tau(x)\, dx$ for all  $\tau\in\mathcal D(\Omega)$. Our notation is close to (Vladimirov~\cite{vVladimirov}).

%_____________________
\section{Notational Bridge to Colombeau Theory}\label{S: Notational Bridge to Colombeau Theory}

	This article does not necessarily require a background in Colombeau's theory (known also as a ``non-linear theory of generalized functions''). However, for those who already are familiar with Colombeau~\cite{jCol84a}-\cite{jCol91} theory, we present a notational comparison which might facilitate the reading of the rest of the article. Notice that on the right hand side of the list below the symbols are borrowed either from \emph{Robinson's non-standard analysis} (associated with the field $^*\mathbb R$ of non-standard real numbers: Robinson~\cite{aRob66}, Lindstr\o m~\cite{tLin}, Cavalcante~\cite{rCav}) \emph{or} \emph{Robinson non-standard asymptotic analysis} (associated with the field $^\rho\mathbb R$ of real asymptotic numbers: Robinson~\cite{aRob73},  Lightstone \& Robinson~\cite{LiRob},  Luxemburg~\cite{wLux}, Pestov~\cite{vPestov}, Todorov \& Wolf~\cite{TodorovWolf}). No background in non-standard analysis, nor in non-standard asymptotic analysis is expected from the reader; some readers might decide to skip this section altogether and proceed with the next. Those who browse through the list of comparisons should not take the comparison literary: for example, we look upon $^*\mathbb R$ as a \emph{refinement of} $\mathbb R^{(0, 1]}$ (similarly, say, the Lebesgue integral is a \emph{refinement} of the Riemann integral).
\begin{enumerate}
\item $\mathbb R^{(0, 1]}\quad  >>>>>>>>>>\; ^*\mathbb R$.
\item $\widetilde{\mathbb R}\quad \quad\; >>>>>>>>>\quad {^\rho\mathbb R}=\widehat{\mathbb R}$.
\item $\varepsilon\quad \quad\quad >>>>>>>>>\quad \rho$. 
\item $[\varepsilon]\quad \quad\quad >>>>>>>>>\quad \widehat{\rho}=s$.
\item $\widetilde{\mathbb C} \quad \quad\; >>>>>>>>>>\quad \widehat{\mathbb C}=\widehat{\mathbb R}(i)$.
\item $\widetilde{\Omega}_c \quad \quad\; >>>>>>>>>>\quad \widehat{\mu}(\Omega)=\big\{r+h: r\in\Omega, h\in\widehat{\mathbb R}^d,\; |h|\approx0\big\}$.

\item $\mathcal C^\infty(\Omega) \quad \; >>>>>>>>>>\quad  \mathcal E(\Omega)=\mathcal C^\infty(\Omega)$.
\item $\mathcal C^\infty(\widehat{\Omega}_c,\; \widetilde{\mathbb C}) \quad \; >>>>>>>>>>\quad \mathcal C^\infty(\widehat{\mu}(\Omega),\; \widehat{\mathbb C})$.
\item $\mathcal G(\Omega) \quad \; >>>>>>>>>>\quad \widehat{\mathcal E(\Omega)}$.

\item $\mathcal G^\infty(\Omega) \quad \; >>>>>>>>>>\quad \mathcal R^\infty(\Omega)$.

\item $\mathcal G^0(\Omega) \quad \; >>>>>>>>>>\quad \widehat{\mathcal F(\Omega)}$.
\end{enumerate}
%______________________________
\section{Background from Algebra}\label{S: Background from Algebra}

	We review briefly some basic definitions from algebra related to totally ordered (real) fields. We also present several examples of totally ordered \emph{non-Archimedean fields} (fields with non-zero infinitesimals). We complete this section with two uniqueness results about totally ordered fields.

	For general references on the topic we refer to: Dales \& Woodin~\cite{DalWoodin}, Lang (\cite{sLang}, Chapter XI), Waerden~(\cite{VanDerWaerden}, Chapter 11), Zariski \& Pierre Samuel~\cite{ZarSam}. For references about valuation fields we refer to: Bourbaki (\cite{nBourbaki90} Chapter VI), Lightstone~\&~Robinson~(\cite{LiRob}, Chapter 1) and Ribenboim~\cite{pRib}. 

	Readers without preliminary experience with \emph{infinitesimals} (beyond the context of \emph{history of calculus}) are strongly encouraged to spend some time with one of our simple examples of non-Archimedean fields - say the field $\mathbb R(t^\mathbb Z)$ of Laurent series (below) and try to become familiar with the concept of \emph{non-zero infinitesimal} elements (treated here as \emph{generalized numbers}) before proceeding to the next section. For those who are experiencing philosophical doubts about the mere \emph{rights of the infinitesimals to exist}, we should mention that every totally ordered field - which contains a proper copy of $\mathbb R$ - must also contain non-zero infinitesimals.  

	Here is our brief excursion into algebra: 
 
\begin{enumerate}
\item A field $\mathbb K$ is \textbf{orderable} if there exists a subset $\mathbb K_+$ of $\mathbb K\setminus\{0\}$ which is closed under addition and multiplication in $\mathbb K$, and is such that, for every $x\in\mathbb K\setminus\{0\}$  either $x\in\mathbb K_+$, or $-x\in\mathbb K_+$. We should mention that a field $\mathbb K$ is orderable \ifff $\mathbb K$ is \textbf{(formally) real} in the sense that for every $n\in\mathbb N$ the equation $\sum_{i=1}^n x^2_i=0$ only admits the trivial solution $x_1=x_2=\dots=x_n=0$ in $\mathbb K^n$. If $\mathbb K$ is orderable, then every set $\mathbb K_+$ defines an \textbf{order relation} on $\mathbb K$ by: $x<y$ if $y-x\in\mathbb K_+$. In this case, we refer to $\mathbb K$ as a \textbf{totally ordered field} and define the absolute value by $|x|=\max\{x, -x\}$. Notice that every totally ordered field contains a copy of the field of rational numbers $\mathbb Q$.
\begin{example}[Non-Orderable] The field of complex numbers $\mathbb C$ and the field of real $p$-adic numbers $\mathbb Q_p$ (Ingleton~\cite{wIngleton}) are both \textbf{non-orderable}. 
\end{example}
 
\item Let $\mathbb K$ be a \textbf{totally ordered field}. We denote by
\begin{align}
& \mathcal I(\mathbb K)=\big\{x\in\mathbb K: |x|<1/n \text{ for all } n\in\mathbb N \big\},\notag\\
& \mathcal F(\mathbb K)=\big\{x\in\mathbb K: |x|\leq n \text{ for some } n\in\mathbb N \big\},\notag\\
& \mathcal L(\mathbb K)=\big\{x\in\mathbb K: n<|x|\text{ for all } n\in\mathbb N \big\},\notag
\end{align}
the sets of \textbf{infinitesimal, finite and infinitely large} elements in $\mathbb K$, respectively. It is customary to write $x\approx 0$ instead of $x\in\mathcal I(\mathbb K)$.  We have $\mathcal F(\mathbb K)\cup\mathcal L(\mathbb K)=\mathbb K$,  $\mathcal F(\mathbb K)\cap\mathcal L(\mathbb K)=\varnothing$, $0\in\mathcal I(\mathbb K)\subset\mathcal F(\mathbb K)$, $\mathbb Q\subseteq\mathcal F(\mathbb K)$ and $\mathbb Q\cap\mathcal I(\mathbb K)=\{0\}$. Also, if $x\in\mathbb K\setminus\{0\}$, then $x\in\mathcal I(\mathbb K)$ \ifff $1/x\in\mathcal L(\mathbb K)$. Finally, $\mathcal F(\mathbb K)$ is an integral domain and $\mathcal I(\mathbb K)$ is its unique convex maximal ideal. Consequently, the residue ring $^{\st}\mathbb K=:\mathcal F(\mathbb K)/\mathcal I(\mathbb K)$ is a totally ordered subfield of $\mathbb R$. We denote by $\st_\mathbb K:\mathcal F(\mathbb K)\to {^{\st}\mathbb K}$ the canonical homomorphism and observe that $\st_\mathbb K$ preserves the order in the sense that $x<y$ implies $\st_\mathbb K(x)\leq\st_\mathbb K(y)$. If $^{\st}\mathbb K$ is a subfield of $\mathbb K$, we refer to the mapping $\st_\mathbb K: \mathcal F(\mathbb K)\to\mathbb K$, as a \textbf{standard part mapping}. 

\item A field $\mathbb K$ is \textbf{Archimedean} if\,  $\mathcal I(\mathbb K)=\{0\}$ (or, equivalently, if $\mathbb K=\mathcal F(\mathbb K)$ or if $\mathcal L(\mathbb R)=\varnothing$). Otherwise it is called \textbf{non-Archimedean}.

\begin{example}[Archimedean Fields] $\mathbb R$ is Archimedean, because $\mathcal I(\mathbb R)=\{0\}$. We also have $\mathcal F(\mathbb R)=\mathbb R$ and $\mathcal L(\mathbb R)=\varnothing$. Actually, a totally ordered field is Archimedean \ifff it is a subfield of $\mathbb R$. 
\end{example}
\item  Every totally ordered field $\mathbb K$ which contains a proper copy of $\mathbb R$ is non-Archimedean. For example, the \textbf{field $\mathbb R(t^\mathbb Z)$ of Laurent series} with real coefficients   is non-Archimedean. Here $\sum_{n=0}^\infty\;  a_n\, t^{m+n} >0 \text{ if } a_0>0$. Also, we define the embedding $\mathbb Q\emb\mathbb R(t^\mathbb Z)$ by $q\to qt^0+0 t+0 t^2+\dots$. The elements: $t, t^2, \dots t^n, t+t^2$, $\sum_{n=1}^\infty\, n!\, t^n$, etc. are non-zero infinitesimals; $5, 5+t, 5+t^2$, $\sum_{n=0}^\infty\, n!\, t^n$, etc. are finite, but not infinitesimal and $1/t, 1/t+5+t$, $\sum_{n=-1}^\infty\, n!\, t^n$, etc. are infinitely large. Let us show that, for example, $t$ is a positive infinitesimal in $\mathbb R(t^\mathbb Z)$. It is clear that $t>0$ since $a_0=1>0$. We have to show that that $t<1/n$ for all $n\in\mathbb N$. Indeed,  $1/n-t=(1/n)t^0+(-1)t+0 t^2+0 t^3+\dots>0$ since $a_0=1/n>0$. 

\begin{examples}[Non-Archimedean Fields] Here are several examples of non-Archimedean fields.
\begin{Ex-enum}
\item The field $\mathbb R(t)$ of rational functions with real coefficients.
\item The field $\mathbb R(t^\mathbb Z)$ of Laurent series with real coefficients (mentioned above).

\item The field $\mathbb R\bra t^\mathbb R\ket$ of the Levi-Civit\'{a} series with real coefficients (Levi-Civit\'{a}~\cite{tLeviCivita}).

\item  The field $\mathbb R((t^\mathbb R))$ of the Hahn series with real coefficients and valuation group $\mathbb R$ (Hahn~\cite{hHahn}). 

\item Any field $^*\mathbb R$ of Robinson non-standard real numbers (Robinson~\cite{aRob66}), 

\item Any field $^\rho\mathbb R$ of Robinson's asymptotic real numbers (Robinson~\cite{aRob73}).
\end{Ex-enum}
\end{examples} 

\item A field $\mathbb K$ (not necessarily ordered by presumption) is \textbf{real closed} if 
\begin{enumerate}
\item For every $a\in\mathbb K$ at least one of the equations $x^2=a$ or $x^2=-a$ has a solution in $\mathbb K$. 
\item Every polynomial $P\in\mathbb K[x]$ \emph{of odd degree} has a root in $\mathbb K$. 
\end{enumerate}
\begin{examples} $\mathbb R$, $\mathbb R\bra t^\mathbb R\ket$, $\mathbb R((t^\mathbb R))$, $^*\mathbb R$ and $^\rho\mathbb R$ are real closed. The fields $\mathbb R(t)$ and $\mathbb R(t^\mathbb Z)$ are not real closed. 
\end{examples}
\item Every real closed field $\mathbb K$ is \textbf{orderable in a unique way} by: $x\geq 0$ in $\mathbb K$ if $x=y^2$ for some $y\in\mathbb K$. 
\item A totally ordered field $\mathbb K$ is \emph{real closed} \ifff $\mathbb K(i)$ is \emph{algebraically closed}. 

\item Let $\kappa$ be an infinite cardinal and let $\mathbb K$ be a totally ordered field with $\card(\mathbb K)=\kappa^+$, where $\kappa^+$ is the successor of $\kappa$ (thus $\mathbb K\not=\mathbb Q$). Then:  
\begin{enumerate}
\item $\mathbb K$ is \textbf{Cantor complete} if every family $\{[a_\gamma, b_\gamma]\}_{\gamma\in\Gamma}$ of bounded closed intervals with the f.i.p. (finite intersection property) and $\card(\Gamma)\leq\kappa$ has a non-empty intersection $\bigcap_{\gamma\in\Gamma}\, [a_\gamma, b_\gamma]\not=\varnothing$.

\item $\mathbb K$ is \textbf{algebraically saturated} if every family $\{(a_\gamma, b_\gamma)\}_{\gamma\in\Gamma}$ of open intervals (bounded or not) with the f.i.p. and $\card(\Gamma)\newline\leq\kappa$ has a non-empty intersection $\bigcap_{\gamma\in\Gamma}\, (a_\gamma, b_\gamma)\not=\varnothing$. 
\end{enumerate}
	It is clear that every algebraically saturated field is also Cantor complete. Here are some examples.
\begin{examples}[Cantor Complete and Saturated] The fields $\mathbb R$ is \emph{Cantor complete} (assuming that $\card(\mathbb N))^+=\card(\mathbb R)$). The fields $^*\mathbb R$, $^\rho\mathbb R$ are also \emph{Cantor complete} (assuming that  $^*\mathbb R$ is a $\card(^*\mathbb R)$-saturated non-standard extension of $\mathbb R$ in the sense of the non-standard analysis, Lindstr\o m~\cite{tLin}). The fields $\mathbb R(t)$, $\mathbb R(t^\mathbb Z)$, $\mathbb R\bra t^\mathbb R\ket$ and $\mathbb R((t^\mathbb R))$ are \emph{not Cantor complete}. 
Also,  $^*\mathbb R$ is \emph{algebraically saturated}. The fields $\mathbb R$, $\mathbb R(t)$, $\mathbb R(t^\mathbb Z)$, $\mathbb R\bra t^\mathbb R\ket$, $\mathbb R((t^\mathbb R))$ and $^\rho\mathbb R$  are \emph{not saturated}. 
\end{examples}
\end{enumerate}

	The next result requires the generalized continuum hypothesis (GCH) $2^\kappa=\kappa^+$ in addition to the more conventional ZFC set-theoretical framework.

\begin{theorem}[First Uniqueness Result]\label{T: First Uniqueness Result} All real closed algebraically saturated fields of the same cardinality are isomorphic. Consequently, for every infinite cardinal there is unique, up to isomorphism, non-standard extension $^*\mathbb R$ of $\mathbb R$ such that $\card(^*\mathbb R)=\kappa^+$.
\end{theorem}
\begin{proof} We refer to Erd\"{o}s \&  Gillman \& Henriksen~\cite{ErdGillHenr55}
(for a presentation see also: Gillman \& Jerison~\cite{lFux63}, p. 179-185).
\end{proof}

	The next uniqueness result essentially involves the previous uniqueness result (and thus it also cannot  survive without the GCH: $2^\kappa=\kappa^+$).

\begin{theorem}[Second Uniqueness Result] \label{T: Second Uniqueness Result} 	Let $\kappa$ be an infinite cardinal and $\mathbb K$ be a totally ordered field with the following properties:
\begin{T-enum}

\item $\card(\mathbb K)=\kappa^+$.

\item $\mathbb K$ is \textbf{Cantor complete real closed field}. 

\item $\mathbb K$ \textbf{contains} $\mathbb R$ as a subfield.

\item $\mathbb K$ admits an \textbf{infinitesimal scale}, i.e. there exists a positive infinitesimal $s$ in $\mathbb K$ such that the sequence $(s^{-n})$ is unbounded (from above). 

\end{T-enum}
Then (for the fixed $\kappa$) the field $\mathbb K$ is unique up to field isomorphism.
\end{theorem}
\begin{proof} We refer to Todorov \& Wolf~\cite{TodorovWolf} (where $\mathbb K$ appears as $^\rho\mathbb R$). For a presentation see also (Todorov~\cite{tdTodorovAxioms11}, Section 3). 
\end{proof} 
	Notice that (iii) and (iv) imply that $\mathbb R$ is a proper subfield of $\mathbb K$ and (iv) implies that $\mathbb K$ is non-saturated. 

%_____________________________
\section{Generalized Scalars in Axioms}\label{S: Generalized Scalars and Functions in Axioms}
We describe a field of generalized numbers $\widehat{\mathbb{R}}$ along with its complex companion $\widehat{\mathbb{C}}=\widehat{\mathbb{R}}(i)$ by means of several axioms. 
What follows is a modification and simplification of some results in Todorov~\cite{tdTodorovAxioms11}.
\begin{description}

\item[Axiom 1 (Cardinality Principle)] $\card(\widehat{\mathbb{R}})=\mathfrak{c}^+$. 
\item[Axiom 2 (Transfer Principle)] $\widehat{\mathbb{R}}$ is a \textbf{real closed Cantor complete field}.

\end{description}

\begin{description}

\item[Axiom 3 (Extension Principle)] $\widehat{\mathbb{R}}$ contains $\mathbb{R}$ as a subfield, i.e. $\mathbb{R}\subseteq\widehat{\mathbb{R}}$.
\end{description}

\begin{description}
\item[Axiom 4 (Scale Principle)] $\widehat{\mathbb{R}}$ admits an \textbf{infinitesimal scale}, i.e. there exists a positive infinitesimal $s$ in $\widehat{\mathbb{R}}$ such that the sequence $(s^{-n})$ is unbounded (from above). 
\end{description}

	Notice that the exponents $s^q$ are well defined in $\widehat{\mathbb{R}}$ for all $q\in\mathbb{Q}$ since $\widehat{\mathbb R}$ is a real closed field. In addition we impose the following:

\begin{description}
\item[Axiom 5 (Exponentiation Principle)] $\widehat{\mathbb{R}}$ admits {\bf exponentiation} in the sense that for every infinitesimal scale $s\in\widehat{\mathbb R}$ there exists a strictly decreasing function $\exp_s: \mathcal F(\widehat{\mathbb{R}})\to\widehat{\mathbb{R}}_+$ which is a group isomorphism from $(\mathcal F(\widehat{\mathbb{R}}), +)$ onto $(\widehat{\mathbb{R}}_+, \cdot)$ such that $(\forall q\in\mathbb{Q})(\exp_s(q)=s^q)$. We shall often write $s^x$ instead of $\exp_s(x)$.
\end{description}

	Notice that the inverse $\log_s: \widehat{\mathbb{R}}_+\to\mathcal F(\widehat{\mathbb{R}})$ of $\exp_s$ exists and $\ln{s}=1/\log_s{e}$.

%_____________________________

\section{Uniqueness and Existence of Generalized Scalars}\label{S: Uniqueness and Existence of Generalized Scalars}

	We show both the existence and uniqueness of the fields $\widehat{\mathbb R}$ and $\widehat{\mathbb C}=\widehat{\mathbb R}(i)$.

\begin{theorem}[Uniqueness of Scalars]\label{T: Uniqueness of Scalars} If there exists a field $\widehat{\mathbb{R}}$ satisfying Axiom 1-5, then $\widehat{\mathbb{R}}$ is unique up to a field isomorphism. Consequently, the field $\widehat{\mathbb{C}}=:\widehat{\mathbb{R}}(i)$ is also uniquely determined by Axiom 1-5 up to a field isomorphism.
\end{theorem}
\begin{proof} This is a particular case of Theorem~\ref{T: Second Uniqueness Result} for $\kappa=\frak{c}$. 
\end{proof}

	Let $^*\mathbb R$ be a field of non-standard real numbers (Robinson\cite{aRob66},  Lindstr\o m~\cite{tLin}). Let $\rho$ be a positive infinitesimal in $^*\mathbb R$. The \textbf{Robinson field of $\rho$-asymptotic numbers} is defined  by $^\rho\mathbb R=\mathcal M_\rho/\mathcal N_\rho$, where   
\begin{align}
& \mathcal M_\rho=\big\{\xi\in{^*\mathbb R}: |\xi|\leq\rho^{-n} \text{ for some } n\in\mathbb N\big\},\notag\\
&\mathcal N_\rho=\big\{\xi\in{^*\mathbb R}: |\xi|<\rho^{n} \text{ for all } n\in\mathbb N\big\},\notag
\end{align}
(Robinson~\cite{aRob73},  Lightstone \& Robinson~\cite{LiRob}). We denote by $\widehat{\xi}$ the equivalence class of $\xi\in\mathcal M_\rho$.

\begin{theorem}[Existence of Scalars]\label{T: Existence of Scalars}  Let $^*\mathbb R$ be a field of non-standard real numbers with $\card(^*\mathbb R)=\frak{c}^+$ (Cavalcante~\cite{rCav}). Then:
\begin{T-enum}
\item $^\rho\mathbb R$ satisfies Axioms 1-5 ($^\rho\mathbb R$ is a model for these axioms). 

\item $^\rho\mathbb R$ and $\widehat{\mathbb R}$ are isomorphic and $\widehat{\rho}$ is a scale for $\widehat{\mathbb R}$. 
\end{T-enum}
Consequently, both $\widehat{\mathbb R}$ and $\widehat{\mathbb C}=\widehat{\mathbb R}(i)$ exist.
\end{theorem}

\begin{proof}
\begin{Pf-enum}

\item  We leave it to the reader to verify that $^\rho\mathbb R$ satisfies Axioms 1-4. Here is a simple proof that $^\rho\mathbb R$ satisfies Axiom 5 (treated as a theorem): The (internal) function $f: {^*\mathbb R}\to{^*\mathbb R}$, $f(\xi)=\rho^\xi=e^{\xi\ln{\rho}}$, is well-defined for all $\xi\in{^*\mathbb R}$ by the Transfer Principle (see Robinson~\cite{aRob66} or/and Lindstr\o m~\cite{tLin}) and $f^\prime(\xi)=\rho^\xi\ln{\rho}$. Let $\xi, \eta\in\mathcal F(^*\mathbb R)$. Next, we show that $\xi-\eta\in\mathcal N_\rho$ implies $\rho^\xi-\rho^\eta\in\mathcal N_\rho$. Indeed, by the mean value theorem (applied again, by transfer) $|\rho^\xi-\rho^\eta|=|\rho^{\xi+t(\eta-\xi)}\ln{\rho}|\,|\xi-\eta|$ for some $t\in{^*(0, 1)}$. We observe that $\rho^{\xi+t(\eta-\xi)}\ln{\rho}\in\mathcal M_\rho$ since $\xi+t(\eta-\xi)$ is finite and thus $\rho^\xi-\rho^\eta\in\mathcal N_\rho$. The latter makes the definition $s^{\widehat{\xi}}= \widehat{\rho^\xi}$ correct for all $\xi\in\mathcal F(^*\mathbb R)$, where $s=\widehat{\rho}$. We also observe that $\widehat{\xi}$ is finite in $^\rho\mathbb R$ \ifff $\xi$ is finite in $^*\mathbb R$ which completes the definition of the exponent $s^x$ after letting $x=\widehat{\xi}$. We leave it to the reader to verify that $s^x$ is a group isomorphism.
\item  $^\rho\mathbb R$ and $\widehat{\mathbb R}$ are isomorphic by Theorem~\ref{T: Uniqueness of Scalars} since both fields satisfy Axioms 1-5.  
\end{Pf-enum}
\end{proof}
%___________________
\section{Some Properties of the Fields $\widehat{\mathbb{R}}$ and $\widehat{\mathbb{C}}$}
	In addition to the properties of any non-Archimedean field (Section~\ref{S: Background from Algebra}), we have the following:
\begin{lemma} 
\begin{L-enum}
\item $\mathbb R\cap\mathcal I(\widehat{\mathbb R})=\{0\}$. Consequently, $\mathbb C\cap\mathcal I(\widehat{\mathbb C})=\{0\}$.

\item $\widehat{\mathbb R}$ is a \textbf{totally ordered field} under the order relation: $x\geq 0$ in $\widehat{\mathbb R}$ if $x=y^2$ for some $y\in\widehat{\mathbb R}$. We endow $\widehat{\mathbb R}$ with the \textbf{order topology}. 
\item $^\st\widehat{\mathbb R}=\mathbb R$. Consequently,  the \textbf{standard part mapping} $\widehat{\st}: \mathcal F(\widehat{\mathbb R})\to\widehat{\mathbb R}$ is an order preserving ring homomorphism with range $\widehat{\st}[\mathcal F(\widehat{\mathbb R})]=\mathbb R$ (we use the notation $\widehat{\st}$ instead of $\st_{\widehat{\mathbb R}}$). Similarly, $\widehat{\st}: \mathcal F(\widehat{\mathbb C})\to\widehat{\mathbb C}$ is a ring homomorphism with $\widehat{\st}[\mathcal F(\widehat{\mathbb C})]=\mathbb C$.
\end{L-enum}
\end{lemma}	
\begin{proof} 
\begin{Pf-enum}
\item The intersection in $\mathbb R\cap\mathcal I(\widehat{\mathbb R})$ makes sense since $\mathbb R\subset\widehat{\mathbb R}$ by Axiom 3. On the other hand, the only infinitesimal in $\mathbb R$ is the zero since $\mathbb R$ is an Archimedean field.

\item follows directly from the fact that $\widehat{\mathbb R}$ is a real closed field (Axiom 2).

\item  $^\st\widehat{\mathbb R}\subseteq\mathbb R$ holds trivially (it holds for any field). Also, $\mathbb R\subset\widehat{\mathbb R}$ (Axiom 3) implies $\mathbb R\subset\mathcal F(\widehat{\mathbb R})$. Next, if $\widehat{\st}(r)=0$ for some $r\in\mathbb R$, then $r=0$. The latter means $\mathbb R\subseteq{^\st\widehat{\mathbb R}}$ as required. 
\end{Pf-enum}
\end{proof} 

\begin{definition}[$s$-Valuation] Let $s$ be a scale of $\widehat{\mathbb R}$ (Axiom 4). We shall keep $s$ fixed in what follows. 
\begin{D-enum}
\item  We define a valuation $v_s:\widehat{\mathbb{C}}\to\mathbb{R}\cup\{\infty\}$ by $v_s(0)=\infty$ and $v_s(x)=\sup\{q\in\mathbb{Q}: |x|/s^q\approx 0\}$ if $x\not=0$. 
\item We define $d_s: \widehat{\mathbb C}\times\widehat{\mathbb C}\to\mathbb R$ by $d_s(x, y)=e^{-v_s(x-y)}$.
\end{D-enum}
\end{definition}
\begin{theorem}[Properties of Valuation] The mapping $v_s$ is a \textbf{convex non-Archimedean valuation} on $\widehat{\mathbb{C}}$ in the sense that for every $x,y\in\widehat{\mathbb{C}}$ {\em :} 
\begin{T-enum} 
\item $v_s(x)=\infty$ \ifff $x=0$.
\item $v_s(xy)=v_s(x)+v_s(y)$,
\item $v_s(x+y)\geq\min\{v_s(x), v_s(y)\}$.
\item $(|x|<|y|\Rightarrow v_s(x)\geq v_s(y))$.
\end{T-enum}
\end{theorem}
\begin{proof} We leave the verification to the reader.
\end{proof}

\begin{theorem}[Algebraic and Topological Properties]\label{T: Algebraic and Topological Properties} \begin{T-enum}
\item $\widehat{\mathbb{C}}$ is an \textbf{algebraically closed field}. We endow\, $\widehat{\mathbb C}$ with the product topology inherited from the order topology on $\widehat{\mathbb R}$. 

\item Both $(\widehat{\mathbb{R}}, d_s)$ and $(\widehat{\mathbb{C}}, d_s)$ is \textbf{spherically complete ultra-metric spaces} in the sense that every nested sequence of closed balls in $\widehat{\mathbb{R}}$ or in $\widehat{\mathbb{C}}$ has a non-empty intersection. Consequently, both $\widehat{\mathbb{R}}$ and $\widehat{\mathbb{C}}$ are also  \textbf{sequentially complete}. 

\item The product-order topology and the metric topology (sharp topology) coincide on both $\widehat{\mathbb{R}}$ and on $\widehat{\mathbb{C}}$. 
\item Let $(a_n)$ be a sequence in $\widehat{\mathbb{C}}$. Then $\lim_{n\to\infty}a_n=0$  \ifff $\lim_{n\to\infty}v_s(a_n)=\infty$   \ifff $\sum_{n=0}^\infty a_n$ is convergent in $\widehat{\mathbb{C}}$.
\end{T-enum}
\end{theorem}
\begin{proof} (i) follows directly from the fact that $\widehat{\mathbb R}$ is a real closed field (Axiom 2). The properties (ii)-(iv) hold in any ultra-metric space (Lightstone \& Robinson~\cite{LiRob}, Ch.1, \S4). 
\end{proof}
\begin{example} Let $s$ be a scale for $\widehat{\mathbb R}$ and $x$ be a (finite) number in $\widehat{\mathbb R}$ of the form $x=r+h$, where $r\in\mathbb R$ and $v_s(h)>0$. Then $s^x=s^r\sum_{n=0}^\infty\, \frac{(\ln s)^n h^n}{n!}$ (the series is convergent by part (iv) of the above theorem).
\end{example}
%______________________
\section{Generalized Functions in Axioms}\label{S: Generalized Functions in Axioms}

	We define a differential algebra $\widehat{\mathcal E(\Omega)}$ over the field $\widehat{\mathbb C}$ axiomatically. We complete this section with an open problem. In what follows we denote by $\mathcal T^d$ the usual topology on $\mathbb R^d$.

\begin{definition}[$\mathcal C^\infty$-Functions on a Monad] \label{D: Cinfinity-Functions on a Monad} Let $\Omega$ be an open subset of $\mathbb{R}^d$. 
\begin{D-enum}

\item The monad of $\Omega$ in $\widehat{\mathbb{R}}^d$ is  $
\widehat{\mu}(\Omega)=\big\{r+h:
r\in\Omega,\; h\in\widehat{\mathbb{R}}^d, |h|\approx 0\big\}$.
\item We denote by $\mathcal{C}^\infty(\widehat{\mu}(\Omega), \widehat{\mathbb{C}})$ the ring of the $\mathcal{C}^\infty$-functions from $\widehat{\mu}(\Omega)$ to $\widehat{\mathbb{C}}$ (i.e. $\mathcal{C}^\infty(\widehat{\mu}(\Omega), \widehat{\mathbb{C}})$  consists of all functions from $\widehat{\mu}(\Omega)$ to $\widehat{\mathbb{C}}$ with continuous partial derivatives of all orders). Let $\mathcal O$ be an open subset of $\Omega$. We define a \textbf{restriction} $f\rest\mathcal O\in\mathcal{C}^\infty(\widehat{\mu}(\mathcal O), \widehat{\mathbb{C}})$ by $(f\rest\mathcal O)(x)=f(x)$ for $x\in\widehat{\mu}(\mathcal O)$. We denote by $\supp(f)\subseteq\Omega$ the \textbf{support} of $f\in\mathcal{C}^\infty(\widehat{\mu}(\Omega), \widehat{\mathbb{C}})$.
\end{D-enum}
\end{definition}

\begin{theorem}[Algebra of $\mathcal C^\infty$-Functions] 
\begin{T-enum}
\item $\mathcal{C}^\infty(\widehat{\mu}(\Omega), \widehat{\mathbb{C}})$ is a \textbf{differential algebra over  the field}\; $\widehat{\mathbb{C}}$. 
\item The family  $\big\{\mathcal{C}^\infty(\widehat{\mu}(\Omega), \widehat{\mathbb{C}})\big\}_{\Omega\in\mathcal T^d}$ is a \textbf{sheaf of differential algebras over the field} $\widehat{\mathbb{C}}$. 
\end{T-enum}
\end{theorem}

\begin{remark}[Warning]\label{R: Warning} The above  result follows immediately and we leave it to the reader. We should note that neither the \emph{fundamental theorem of calculus}, nor the \emph{mean value theorem} hold in $\mathcal C^\infty(\widehat{\mu}(\Omega),\; \widehat{\mathbb C})$. For example, let $\Omega$ is an open connected subset of $\mathbb R^d$ and $r\in\Omega$. Then the monad $\widehat{\mu}(\Omega)$ is also an open connected subset of $\widehat{\mathbb R}^d$. However, the function $f: \widehat{\mu}(\Omega)\to\widehat{\mathbb C}$, defined by $f(x)=1$ for $x\approx r$, and $f(x)=0$ otherwise, is (obviously) not a constant function, but still $\nabla f=0$ on $\widehat{\mu}(\Omega)$. 
\end{remark}

	Let $\widehat{\mathcal E(\Omega)}$ be a \emph{differential subalgebra of} $\mathcal C^\infty(\widehat{\mu}(\Omega),\; \widehat{\mathbb C})$ over the field $\widehat{\mathbb C}$ satisfying the following axioms:

\begin{description}

\item[Axiom 6 (Constant Functions)] The set of the \emph{constant functions} on $\widehat{\mu}(\Omega)$ coincides exactly with $\widehat{\mathbb C}$. 

\item[Axiom 7 (MVTh)] The \emph{Mean Value Theorem} (from multivariable calculus) holds in $\widehat{\mathcal E(\Omega)}$.

\item[Axiom 8 (Colombeau Embedding)] There exists a \emph{Colombeau type of embedding} $E_\Omega: \mathcal D^\prime(\Omega)\to \widehat{\mathcal E(\Omega)}$, i.e. $E_\Omega$ is an injective  homomorphism of differential vector spaces over $\mathbb C$ such that:
\begin{description} 
\item[(a)] $E_\Omega[\mathcal D^\prime(\Omega)]$ is a differential vector subspace of $\widehat{\mathcal E(\Omega)}$ over $\mathbb C$.
\item[(b)] $(E_\Omega\circ S_\Omega)[\mathcal E(\Omega)]$ is a differential subalgebra of $\widehat{\mathcal E(\Omega)}$ over $\mathbb C$ (see the end of the Introduction).  
\end{description}
\end{description}
\begin{remark}[Notational Simplification] \label{R: Notational Simplification} We often simplify the notation by letting  $(E_\Omega\circ S_\Omega)[\mathcal E(\Omega)]=\mathcal E(\Omega)$ and $E_\Omega[\mathcal D^\prime(\Omega)]= \mathcal D^\prime(\Omega)$ and summarize the above axiom with the inclusions:  $
\mathcal E(\Omega)\subset\mathcal D^\prime(\Omega)\subset\widehat{\mathcal E(\Omega)}$.
\end{remark}

	Let $\mathcal{R}_{c}(\Omega)$ denote the set of all \emph{measurable relatively compact subsets of} $\Omega$. 

\begin{description}
\item [Axiom 9 (Integral)]   There exists a mapping  $\widehat{I}: \widehat{\mathcal E(\Omega)}\times\mathcal{R}_{c}(\Omega)\to\widehat{\mathbb C}$ such that:

\begin{description}
\item[(a)]  $\widehat{I}$ is \textbf{$\widehat{\mathbb C}$-linear} in the first variable. 

\item[(b)] $\widehat{I}$ is \textbf{additive} in the second variable in the sense that \emph{if} $\{\Omega_\gamma\}_{\gamma\in\Gamma}$ is a family of pairwise disjoint sets in $\mathcal{R}_{c}(\Omega)$ such that $\bigcup_{\gamma\in\Gamma} \Omega_\gamma\in\mathcal{R}_{c}(\Omega)$ and \emph{if} $\sum_{\gamma\in\Gamma}\;\widehat{I}(\Omega_\gamma, f)$ is convergent in $\widehat{\mathbb C}$, \emph{then} 
\[
\widehat{I}(\bigcup_{\gamma\in\Gamma} \Omega_\gamma, f)=\sum_{\gamma\in\Gamma}\;\widehat{I}(\Omega_\gamma, f).
\]
\item[(c)] Let $T\in\mathcal D^\prime(\Omega)$ and $\tau\in\mathcal D(\Omega)$. Then $\widehat{I}(E_\Omega(T), \mathcal O)=\bra T\, |\, \tau\ket$ for all open $\mathcal O\in\mathcal{R}_{c}(\Omega)$ such that either $\supp(T)\subset\mathcal O$ or $\supp(\tau)\subset\mathcal O$.
\end{description}
\end{description}

\begin{example} We shall often write simply $\int_\mathcal O\, f(x)\, dx$ instead of the more precise $\widehat{I}(f,\mathcal O)$. For example, for every  $\tau\in\mathcal D(\Omega)$ we have $\int_{\mathbb R^d}\, \delta(x)\tau(x)\, dx=\tau(0)$, a notation used often by non-mathematicians. 
\end{example}

\begin{theorem}[Fundamental Theorem of Calculus] Let $\Omega$ be an open connected subset of $\mathbb R^d$ and let $\nabla f(x)=0$ for some $f\in\widehat{\mathcal E(\Omega)}$ and all $x\in\mu(\Omega)$. Then $f(x)=c$ for some $c\in\widehat{\mathbb C}$ and all $x\in\mu(\Omega)$. 
\end{theorem}

\begin{proof} An immediate consequence of the Axiom 7.
\end{proof}

%_____________________
\begin{definition}[Regular Algebra]  A differential subalgebra $\mathcal R(\Omega)$ of $\widehat{\mathcal E(\Omega)}$ is called \emph{regular} if:
\begin{description}

\item[(a)]  $(E_\Omega\circ S_\Omega)[\mathcal E(\Omega)]$ is a differential subring of $\mathcal R(\Omega)$.

\item[(b)] $\mathcal R(\Omega)\cap E_\Omega[\mathcal D^\prime(\Omega)]=(E_\Omega\circ S_\Omega)[\mathcal E(\Omega)]$.
\end{description}
If a regular algebra $\mathcal R(\Omega)$ is maximal (under inclusion), we shall write $\mathcal R^\infty(\Omega)$ instead of $\mathcal R(\Omega)$.
\end{definition}

	It is clear that $(E_\Omega\circ S_\Omega)[\mathcal E(\Omega)]$ is a regular subalgebra of $\widehat{\mathcal E(\Omega)}$ (in a trivial way).

\begin{description}
\item[Axiom 10 (Regularity)]\label{A: Regularity} There exists a regular subalgebra $\mathcal R(\Omega)$ of $\widehat{\mathcal E(\Omega)}$ which is a proper extension of $(E_\Omega\circ S_\Omega)[\mathcal E(\Omega)]$, in symbol, \newline $(E_\Omega\circ S_\Omega)[\mathcal E(\Omega)]\subsetneqq\mathcal R(\Omega)$.
\end{description}

	In the simplified notation (Remark~ \ref{R: Notational Simplification}) the above conditions can be written simply as: $\mathcal E(\Omega)\subseteq\mathcal R(\Omega)$, $\mathcal R(\Omega)\cap \mathcal D^\prime(\Omega)=\mathcal E(\Omega)$ and $\mathcal E(\Omega)\subsetneqq\mathcal R(\Omega)$, respectively.

\begin{description}
\item[Axiom 11 (Maximality Principle)] The algebra $\widehat{\mathcal E(\Omega)}$ is a \textbf{maximal in} $\mathcal C^\infty(\widehat{\mu}(\Omega),\; \widehat{\mathbb C})$  in the sense that there is no a differential subalgebra of $\mathcal C^\infty(\widehat{\mu}(\Omega),\; \widehat{\mathbb C})$ over $\widehat{\mathbb C}$ which is a proper extension of $\widehat{\mathcal E(\Omega)}$ and  which also satisfies Axioms 6-10.
\end{description}
	
%______________________

\begin{theorem}[Uniqueness] Axioms 5-11 determines $\widehat{\mathcal E(\Omega)}$ uniquely up to isomorphism (of differential algebras over the field $\widehat{\mathbb C}$ satisfying Axioms 6-11).
\end{theorem}
\begin{proof} This is an \emph{open problem} in our algebraic (axiomatic) approach.
\end{proof}

%_______________________________
\section{Consistency of All Axioms}\label{S: Consistency of All Axioms}
We already show that Axioms 1-5 are consistent (Theorem~\ref{T: Existence of Scalars}). Here we show that Axioms 1-10 are also consistent by way of a model. 

\begin{theorem}[Consistency]\label{T: Consistency} Axiom 1-10 are consistent (under ZFC and the generalized continuum hypothesis $2^\mathfrak{c}=\mathfrak{c}^+$, where $\mathfrak{c}=\card(\mathbb{R})$).
\end{theorem}

\begin{proof}[Proof 1 (Within Non-Standard Analysis)]  For  a model of Axioms 1-10 in the framework of Robinson's non-standard analysis we refer to (Oberguggenberger \& Todorov~\cite{OberTod98}). 
\end{proof}
\begin{proof}[Proof 2 (Within Standard Analysis)] For  a model of Axioms 1-10 in the framework of standard analysis we refer to (Todorov \& Vernaeve~\cite{TodVern08}).  For a discussion of the same model we also refer to (Todorov~\cite{tdTodorovAxioms11}, \S 4). 

	In both models (mentioned in Proof 1 and Proof 2)
{\small\[
\mathcal R(\Omega)=\{f\in \widehat{\mathcal E(\Omega)}: (\forall x\in\widehat{\mu}(\Omega))(\forall\alpha\in\mathbb N_0^d)(\exists n\in\mathbb N)(|\partial^\alpha f(x)|\leq n) \},
\]}
offers an example for a non-trivial regular subalgebra of $\widehat{\mathcal E(\Omega)}$ over the field $\mathbb C$ (which guarantees the consistency of Axiom 10 with the rest of the axioms).
\end{proof}
%____________________________________
\section{Partial Independence of Axioms}\label{S: Partial Independence of Axioms}

	Like most of the systems of axioms in mathematics, our axioms are not independent. For example, the formula $s^{q+h}=s^q\sum_{n=0}^\infty\, \frac{(\ln s)^n h^n}{n!}$ holds for all $q\in\mathbb Q$ and all $h\in\widehat{\mathbb R}$ with $v_s(h)>0$, without the help of Axiom 5. Indeed, $s^q$ is well defined simply because $\widehat{\mathbb R}$ is a real closed field by Axiom 2 and the series is convergent in $\widehat{\mathbb R}$ because $v_s(\frac{(\ln s)^n h^n}{n!})=v_s(h)\,n\to \infty$ which is a consequence of Axiom 4. It is not worth the effort to isolate an independent system of axioms which would result in a considerable complication of the language (it is not accidental that the axioms in the axiomatic definition of ``group'', for example, are not independent, nor are the axioms for $\mathbb R$). Instead of trying to isolate independent axioms, in this section we shall restrict ourselves to the less ambitious task of showing that some particular subsets of our axioms - with one of the axioms replaced by its negation - are consistent. We refer to this process as a \emph{partial independence}. The purpose is to create a feeling for the role of each axiom relative to the rest of the axioms. 

	If $P$ is a proposition, then $\neg P$ stands for \emph{ its negation}.
\begin{theorem}[Partial Independence] The following subsets of Axiom 1-10 are consistent:
\begin{T-enum}
\item $\{\neg \text{Axiom } 1, \text{Axiom }2, \text{Axiom }3, \text{Axiom }4, \text{Axiom }5\}$ are consistent. 
\item $\{\text{Axiom } 1, \text{Axiom }2, \text{Axiom }3,\; \neg\text{Axiom }4, \text{Axiom }5\}$ are consistent. 
\item $\{\text{Axiom } 6,\; \neg\text{Axiom }7, \text{Axiom }8,\; \text{Axiom }9, \text{Axiom }10\}$ are consistent. 
\end{T-enum}
\end{theorem}
\begin{proof}
\begin{Pf-enum}
\item Let $^\rho\mathbb R$ be the Robinson field of asymptotic numbers generated from $^*\mathbb R=\mathbb R^\mathbb N/\mathcal U$, where $\mathcal U$ is a free ultrafilter on $\mathbb N$ (Lindstr\o m~\cite{tLin}). Then $^\rho\mathbb R$ satisfies Axioms 2-5, but not Axiom 1, because $\card(^\rho\mathbb R)= \frak{c}$ (not $\frak{c}^+$). 

\item Let  $^*\mathbb R=\mathbb R^{\mathbb R_+}/\mathcal U$ for some $\frak{c}^+$-good ultrafilter $\mathcal U$ on $\mathbb R_+$ (Cavalcante~\cite{rCav}). Then $^*\mathbb R$ satisfies Axioms 1-3 and Axiom 5, but not Axiom 4, because $^*\mathbb R$ is algebraically saturated (and thus it does not have a scale $s$). 

\item The algebra $\mathcal{C}^\infty(\widehat{\mu}(\Omega), \widehat{\mathbb{C}})$ (Definition~\ref{D: Cinfinity-Functions on a Monad}) satisfies Axioms 6, Axiom 8-10, but not Axiom 7 (see Remark~\ref{R: Warning}). 
\end{Pf-enum}
\end{proof}

\begin{remark} [A Memory of Hebe Biagioni] Notice that the axioms which define the field of scalars $\widehat{\mathbb R}$ and $\widehat{\mathbb C}$ (Axioms 1-5) do not involve the dimension of the domain $\Omega$ of the generalized functions in $\widehat{\mathcal E(\Omega)}$ (Axioms 6-10). This was not the case however, in the original Colombeau~\cite{jCol84a} construction, where the ring $\overline{\mathbb R}(\Omega)$ of generalized scalars was defined as a particular subring of the ring $\mathcal G(\Omega)$ of generalized functions, and thus depends on the dimension of the domain $\Omega\subseteq\mathbb R^d$. This reminds me of our dear colleague and friend Hebe Biagioni~\cite{aBiag90}, who more than 20 years ago, modified the original Colombeau definition of $\mathcal G(\Omega)$ to achieve independence of the ring of generalized scalars from the dimension $d$ (as it should be with any set of scalars).
\end{remark}
%_______________________________
\textbf{Acknowledgement:} This article evolved from talks presented in the summer of 2013, first in the Bulgarian Academy of Sciences in Sofia for theoretical physicists, and later in the University of Vienna (Austria) at a workshop on \emph{Non-Archimedean Geometry and Analysis}, September 2-6, 2013. The author thanks the organizers: Ivan Todorov (Sofia) and Paolo Giordano and Michael K\"{u}nzinger (Vienna) and all participants, for the invitation and stimulating discussion.  
%_________________

%_______________________

\end{document}